\documentclass[a4paper,11pt]{amsart}

\usepackage[left = 3.5 cm, right = 3.5 cm]{geometry}
\usepackage{amsthm,amssymb,amsmath,amsfonts,mathrsfs,amscd,amsbsy,dsfont,verbatim}
\usepackage{stmaryrd}
\usepackage{graphicx}

\usepackage[all,cmtip]{xy}
\usepackage{comment}
\usepackage{longtable}
\usepackage{marginnote}
\usepackage{babel}
\usepackage{latexsym}
\usepackage{mathtools}
\usepackage[utf8]{inputenc}
\usepackage{xcolor}
\usepackage{fancyhdr}
\usepackage{indentfirst}
\usepackage{paralist}
\usepackage{hyperref}
\usepackage{enumerate}
\usepackage{multicol}

\usepackage{todonotes}

\usepackage[title]{appendix}
\usepackage{url}

\newtheorem{theorem}{Theorem}[section]
\newtheorem{proposition}[theorem]{Proposition}

\newtheorem{corollary}[theorem]{Corollary}

\newtheorem{question}[theorem]{Question}
\newtheorem*{theorem*}{Theorem}

\theoremstyle{definition}
\newtheorem{definition}{Definition}[section]
\newtheorem{remark}[theorem]{Remark}
\newtheorem{notation}[definition]{Notation}

\newcommand{\Q}{\mathbb Q}
\newcommand{\Z}{\mathbb Z}
\newcommand{\F}{\mathbb F}
\newcommand{\p}{\mathfrak{p}}
\newcommand{\f}{\mathfrak{f}}
\newcommand{\e}{\mathfrak{e}}
\newcommand{\GL}{\text{GL}}
\DeclareMathOperator{\Gal}{Gal}
\DeclareMathOperator{\Aut}{Aut}
\DeclareMathOperator{\End}{End}
\DeclareMathOperator{\Frob}{Frob}
\renewcommand{\bar}{\overline}
\DeclareMathOperator{\Tors}{Tors}

\newcommand{\mat}[4]{\left(\begin{array}{cc}
#1 &#2\\
#3 &#4
\end{array}\right)}

\title{Torsion points on \(\GL_2\)-type abelian varieties}
\date{\today}

\author[J. Alessandrì]{Jessica Alessandrì}
\address{Max Planck Institute for Mathematics\\
Vivatsgasse 7\\
53111 Bonn\\
Germany}
\email{alessandri@mpim-bonn.mpg.de}
\urladdr{https://sites.google.com/view/jessicaalessandri}
\author[N. Coppola]{Nirvana Coppola}
\address{Dipartimento di Matematica, Università di Padova\\
Via Trieste, 63\\
35121 Padova\\
Italy}
\email{nirvana.coppola@unipd.it}
\urladdr{https://sites.google.com/view/nirvanacoppola/}


\begin{document}

\begin{abstract}
It is well known that the order of the torsion group of an abelian variety over a number field divides the greatest common divisor of the sizes of points on the reduction at each prime (except for an explicit finite set). Drawing inspiration from Katz's Inventiones paper (1981), we investigate the converse to this for abelian varieties of \(\rm GL_2\)-type and exhibit a tentative list of possible rational torsion orders for modular abelian varieties over \(\mathbb Q\) of dimension up to \(5\).
\end{abstract}

\maketitle

\section{Introduction}\label{sec:intro}
For any elliptic curve \(E\) over \(\Q\), Mazur's celebrated theorem \cite[Theorem 8]{Mazur1977ModularIdeal} classifies the possible isomorphism classes of the torsion subgroup \(\Tors(E)(\Q)\), and in particular shows that \(|\Tors(E)(\Q)| \leq 16\).
Mazur's work was then generalised to other number fields: Merel \cite{Merel1996BornesNombres} proves that, for every number field \(K\), there is a bound for \(|\Tors(E)(K)|\) that depends only on the degree \([K:\Q]\), uniformly over all elliptic curves \(E/K\).

For higher-dimensional abelian varieties, partial results are known in several settings. For a fixed abelian variety over a number field, Le Fourn, Lombardo, and Zywina study bounds as the number field varies \cite{Fourn2022TorsionField}; Clark and Xarles establish bounds over local fields \cite{Clark2008LocalVarieties}. Silverberg also proves bounds depending only on the dimension and the degree of the field of definition for abelian varieties of CM-type \cite{Silverberg1988CMType,Silverberg1988Corrigendum,Silverberg1992FiniteOrder}. 

Uniform bounds are known for certain special classes of abelian varieties: for abelian surfaces with \emph{potential quaternionic multiplication} (meaning that \(\End_{\overline{\Q}}(A)\) is a maximal order in a division quaternion algebra over \(\Q\)), Laga, Schembri, Shnidman, and Voight prove that \(|\Tors(A)(\Q)| \leq 18\) and give explicit lists of candidate torsion subgroups \cite{Laga2024RationalMultiplication}. Under the additional assumption that \(A\) is of \(\GL_2\)-type, they give a complete classification of the possible groups; each group in that classification occurs for infinitely many such abelian surfaces.

In this paper, we work more generally on abelian varieties of \(\GL_2\)-type defined over a number field \(K\), without further restrictions on the dimension.
Our strategy is to use a classical approach in the study of rational points, namely, to address a \emph{local-global problem}.
If \(\p\) is a prime of good reduction for \(A/K\), with residue field \(\F_\p\) of characteristic \(p\), reduction is injective on torsion under the assumption that the absolute ramification index \(\e_\p\) satisfies \(\e_\p<p-1\), which holds with finitely many exceptions (see, for instance, the good-reduction results of Serre and Tate \cite{Serre1968GoodVarieties}, or \cite[Appendix]{Katz1981GaloisVarieties}). For such primes, writing \(N(\p)=|\widetilde A(\F_\p)|\) for the number of points on the reduction of \(A\), we obtain that \(|\Tors(A)(K)|\) divides \(N(\p)\).
A famous question, attributed to Lang, asks whether a converse holds. More precisely, he asked the following.
\begin{question}
    Let \(A\) be an abelian variety over a number field \(K\) and let \(m\ge 2\) be an integer. Suppose that 
    \[
    N(\p) \equiv 0 \pmod m
    \]
    holds for a density-one set of places \(\p\). Does there exist a \(K\)-isogenous abelian variety \(A'\) for which the set of \(K\)-rational torsion points satisfies
    \[
    |\Tors(A')(K)| \equiv 0 \pmod m ?
    \]
\end{question}
Katz \cite{Katz1981GaloisVarieties} answers this question affirmatively for elliptic curves and proves a weaker version for abelian surfaces. More precisely, the set of prime divisors of the torsion order of some abelian surface \(A'\) that is \(K\)-isogenous to \(A\) coincides with the set of primes dividing all the numbers \(N(\p)\). However, he also gives counterexamples in higher dimensions, even to this weaker statement.
Subsequent work has established further counterexamples; see, for example, \cite{Cullinan2007LocalglobalVarieties, Cullinan2010PointsZarhin} and \cite{Cullinan2022DivisibilityFields}.
This problem is related to the so-called \emph{local-global divisibility problem} in algebraic groups (see \cite[\S 8]{Dvornicich2022Local-globalGroups}), which is addressed by Gillibert and Ranieri for abelian varieties of \(\GL_2\)-type in \cite{Gillibert2017OnVarieties, Gillibert2020OnVarieties}.

In this paper, we prove the following theorem, which gives an affirmative answer, in all dimensions, to a weaker but very explicit version of Katz's question for abelian varieties of \(\GL_2\)-type. We also give computational examples showing that the exponent in our theorem is sharp, in the sense that it cannot be increased in general (see \S \ref{sec:app}).
\begin{theorem*}[{See Theorem \ref{thm:main_new}}] 
   Let \(A\) be an abelian variety over a number field \(K\) of \(\GL_2\)-type, with an embedding
   \[
   F \hookrightarrow \End_K(A) \otimes_\Z \Q,
   \]
   where \(F\) is a number field of degree equal to \(\dim(A)\) over \(\Q\). Let \(\ell\) be a rational prime and let \(\lambda\) be a prime of \(F\) over \(\ell\). Suppose that, for an integer \(n \ge 1\), the congruence
    \[
    P_\p^\lambda(1) \equiv 0 \pmod {\lambda^n}
    \]
    holds for a density-one set of finite places \(\p\) of \(K\). Then there exists an abelian variety \(A'\), \(K\)-isogenous to \(A\) through a \(\ell\)-power isogeny, such that 
    \[
    |\Tors(A')(K)| \equiv 0 \pmod {\ell^{\f_{\lambda}n}},
    \]
    where \(\f_\lambda\) is the residue degree of \(\lambda \mid \ell\).
\end{theorem*}

Finally, we show how one can apply our main results to study the torsion of \emph{modular abelian varieties} defined over the rationals, which are associated with weight \(2\) newforms. 
We implement our results in the computational algebra system \texttt{Magma}, in order to produce tentative lists of possible torsion orders for modular abelian varieties of dimension up to \(5\), presented in \S \ref{sec:app}.
The code and the outputs can be found at \cite{Alessandri2026MagmaVarieties}. Although it is not currently known whether the lists are exhaustive, each newform in our dataset comes with a predicted torsion order. As explained in \S \ref{sec:implementation}, our computations provide conjectural predictions for the maximal rational torsion orders in the corresponding \(\Q\)-isogeny classes. Whenever these predictions agree with the values determined by all admissible primes, their validity follows from our theorem. The lists and examples supplement the data currently available in \texttt{LMFDB} \cite{LMFDBDatabase} and in \cite{Nicholls2018DescentCurves}. 
For example, our computation predicts the torsion order \(28\) for the \(\Q\)-isogeny class associated to the level \(39\) newform with \texttt{LMFDB} label \texttt{\href{https://www.lmfdb.org/ModularForm/GL2/Q/holomorphic/39/2/a/b/}{39.2.a.b}}. Provided that the finite set of primes used in the computation gives the correct theoretical minimum, our theorem yields a representative in this \(\Q\)-isogeny class with torsion order \(28\). Such a representative is not listed in \texttt{LMFDB}, and the database contains no abelian surface with torsion order \(28\).

\section{Setting}\label{sec:setting}
In this section, we recall the definition of \(\GL_2\)-type abelian variety and properties of the Galois representations attached to it, in order to reformulate Katz's problem in terms of Galois-stable lattices.

Let \(A\) be an abelian variety of dimension \(g\) defined over a number field \(K\), with fixed algebraic closure \(\bar{K}\). 
\begin{definition}
    We say that \(A\) is of \(\GL_2\)-type if there exists a number field \(F\) of degree \(g\) over \(\Q\), with an embedding
    \[
    F \hookrightarrow \End_K(A) \otimes_\Z \Q.
    \]
\end{definition}

Throughout the paper, whenever we consider an abelian variety of \(\GL_2\)-type, we assume it is equipped with a fixed number field \(F\) and an embedding  \(F \hookrightarrow \End_K(A) \otimes_\Z \Q\).

Let \(\ell\) be a rational prime. One can define the \(\ell\)-adic Galois representation attached to \(A\) as
\[
\rho_\ell : G_K = \Gal(\overline{K}/K) \rightarrow \Aut(T_\ell(A)) \subseteq \Aut(V_\ell(A)),
\]
where:
\begin{itemize}
    \item \(T_\ell(A)\) is the \(\ell\)-adic Tate module of \(A\) (see e.g. \cite{Serre1968GoodVarieties} for the definition);
    \item \(V_\ell(A)= T_\ell(A) \otimes_{\Z_\ell} \Q_\ell\).
\end{itemize}

Since \(A\) is of \(\GL_2\)-type, one has a \(F\)-rational system of compatible representations \((\rho_\lambda)_{\lambda \mid \ell}\) (see \cite{Ribet1976GaloisMultiplications, Lombardo2016ExplicitGL2-varieties, Azon2024ConductorCurves} for the construction), such that $$\rho_\lambda : G_K \rightarrow \GL_2(F_\lambda)$$ is a 2-dimensional representation. The sketch of the construction is as follows: we have \(F_\ell = F \otimes_\Q \Q_\ell = \prod_{\lambda \mid \ell} F_\lambda\), and \(V_\lambda (A) = V_\ell (A) \otimes_{F_\ell} F_\lambda\). The action of \(G_K\) on \(V_\lambda (A)\) is the representation \(\rho_\lambda\). We let \(\f_\lambda\) be the residue degree of \(\lambda\), and \(\mathbb{F}_\lambda\) be the residue field \(\mathcal{O}_{F_\lambda} / \lambda \mathcal{O}_{F_\lambda}\), which has size \(\ell^{\f_\lambda}\).

Moreover, we will need to use the lattice \(T_\lambda (A)= T_\ell (A) \otimes_{\mathcal{O}_{F_\ell}} \mathcal{O}_{F_\lambda}\), where \( \mathcal{O}_{F_\ell} = \mathcal{O}_F \otimes_\Z \Z_\ell\).

\begin{notation}
    Let \(\p\) be a finite place of \(K\), such that \(A\) has good reduction at \(\p\). We denote by \(\Frob_\p\) a Frobenius element in \(G_K\) corresponding to \(\p\); since \(A\) has good reduction at \(\p\), the representations \(\rho_\ell\) and \(\rho_\lambda\) are unramified at \(\p\), thus \(\rho_\ell(\Frob_\p)\) and \(\rho_\lambda(\Frob_\p)\) are well defined.
    We denote by \(P_\p^\ell(t)=\det(1-t\rho_\ell(\Frob_\p))\) the usual characteristic polynomial of Frobenius, and by \(P_\p^\lambda(t)=\det(1-t\rho_\lambda (\Frob_\p))\) the characteristic polynomial of \(\rho_\lambda(\Frob_\p)\).
\end{notation}

\textit{A priori}, \(P_\p^\lambda(t)\) has coefficients in \(F_\lambda\). However, by \cite[Theorem 2.1.2]{Ribet1976GaloisMultiplications}, the system \(\{\rho_\lambda\}_{\lambda}\) is strictly compatible and \(F\)-rational with exceptional set equal to the set of places of \(K\) at which \(A\) has bad reduction, so, with possibly these finitely many exceptions, in fact \(P_\p^\lambda (t) \in F[t]\). Throughout the paper, for \(x \in F_\lambda\) (or \(F\)), we write
\[
x \equiv 0 \pmod{\lambda^n}
\]
whenever the \(\lambda\)-adic valuation of \(x\) is at least \(n\).

\subsection{Katz and Lang's problem}

For a given abelian variety \(A/K\), and \(\p\) a finite place of \(K\) of good reduction for \(A\), we denote by \(\F_\p\) the residue field of \(K\) at \(\p\), by \(\tilde{A}(\F_\p)\) the \(\F_\p\)-points on the reduction of \(A\) modulo \(\p\) and we let \(N(\p) = |\tilde{A}(\F_\p)|\).

In \cite{Katz1981GaloisVarieties}, the following question is proposed: given an abelian variety \(A/K\), and an integer \(m\ge 2\), if 
\[
N(\p) \equiv 0 \pmod m
\]
holds for a set of places \(\p\) of density one, does there exist a \(K\)-isogenous abelian variety \(A'\) for which the set of \(K\)-rational torsion points satisfies
\[
|\Tors(A')(K)| \equiv 0 \pmod m ?
\]

In op.\ cit.\ this problem is reformulated in terms of power-of-a-prime isogenies and the Galois representation \(\rho_\ell\). In particular, the question becomes the following.
Let \(A/K\) be an abelian variety and \(\ell\) a rational prime such that there exists \(n \ge 1\) satisfying
\[
\det(1-\rho_\ell(\gamma)) \equiv 0 \pmod {\ell^n }
\]
for all \(\gamma \in G_K\). Note that, by \v{C}ebotarev Density Theorem \cite[Theorem 13.4]{Neukirch1999AlgebraicTheory}, this condition is equivalent to requiring the congruence for \(\gamma\) of the form \(\Frob_\p\) for a density-one set of places.
Does there exist a pair of \(G_K\)-stable lattices \(\mathcal{L'} \subseteq \mathcal{L} \subseteq V_\ell(A)\) such that the quotient has order \(\ell^n\) and such that \(G_K\) acts trivially on it?

As mentioned, Katz finds a positive answer to this question for the case of elliptic curves \cite[Theorem 2, 2 bis]{Katz1981GaloisVarieties}, and to a weaker version of it for the case of abelian surfaces \cite[Theorem 4]{Katz1981GaloisVarieties}. We adapt Katz's approach for the proof of his Theorem 2 to the case of \(\GL_2\)-type abelian varieties.

\section{Main results}\label{sec:main}

We will use the following result of Katz and Lang \cite[Theorem 1 (bis)]{Katz1981GaloisVarieties} on \(2\)-dimensional representations with coefficients in a local field.

\begin{theorem}\label{thm:Katz}
    Let \(L\) be a field which is complete under a discrete valuation, and whose residue field is finite. Let \(\mathcal{O}_L\) denote the ring of integers in \(L\), and let \(\pi\) denote a uniformising parameter. Let \(V\) be a 2-dimensional \(L\)-vector space, \(G \subseteq \Aut_L(V)\) a compact subgroup, and \(n \ge 1\) an integer. Suppose that we have the congruence
    \[
    \det(1-g) \equiv 0 \mod \pi^n \quad \forall g \in G.
    \]

    Then there exists a \(L\)-basis \(\{e_1,e_2\}\) of \(V\) and integers \(a,b \ge 0\) with \(a+b=n\) such that the elements of \(G\) expressed as matrices with respect to this basis lie in the subgroup
    \[
    \left\lbrace \mat{1 + \pi^a \mathcal{O}_L}{\pi^a \mathcal{O}_L}{\pi^b \mathcal{O}_L}{1 + \pi^b \mathcal{O}_L} \right\rbrace
    \]
    of \(\GL_2(\mathcal{O}_L)\).
\end{theorem}

Katz gives a detailed proof of the case where \(L=\Q_p\) in \cite{Katz1981GaloisVarieties}, but, as he remarks in op. cit., ``mutatis mutandis'', the proof adapts to the general case. 

We are now ready to state our main result.

\begin{theorem}\label{thm:main_new}
    Let \(A\) be an abelian variety over a number field \(K\) of \(\GL_2\)-type. Let \(\ell\) be a rational prime and let \(\lambda\) be a prime of \(F\) over \(\ell\). Suppose that, for an integer \(n \ge 1\), the congruence
    \[
    P_\p^\lambda(1) \equiv 0 \pmod {\lambda^n}
    \]
    holds for a density-one set of finite places \(\p\) of \(K\). Then there exists an abelian variety \(A'\), \(K\)-isogenous to \(A\) through a \(\ell\)-power isogeny, such that 
    \[
    |\Tors(A')(K)| \equiv 0 \pmod {\ell^{\f_{\lambda}n}},
    \]
    where \(\f_\lambda\) is the residue degree of \(\lambda \mid \ell\).
\end{theorem}
\begin{proof}
    With the notation of Theorem \ref{thm:Katz}, let \(L=F_\lambda\), choose (with a small abuse of notation) the uniformiser \(\pi=\lambda\), and then let \(V=V_\lambda (A)\) and \(G=\rho_\lambda(G_K)\). The condition \(\det(1-g)\equiv 0 \pmod {\lambda^n}\) for all \(g \in G\) is equivalent to \(P_\p^{\lambda}(1) \equiv 0 \pmod{\lambda^n}\) for a density-one set of primes \(\p\), by \v Cebotarev Density Theorem. By Theorem \ref{thm:Katz}, there exist two sublattices \(\mathcal{L'} \subseteq \mathcal{L}\) of \(V\), such that \(G_K\) acts trivially on the quotient. Such quotient has order \(\ell^{\f_\lambda n}\), indeed in the basis \(\{e_1,e_2\}\) given by the theorem, \(\mathcal{L} \cong \mathcal{O}_{F_\lambda} e_1 + \mathcal{O}_{F_\lambda} e_2\); \(\mathcal{L}' \cong \mathcal{O}_{F_\lambda} \lambda^a e_1 + \mathcal{O}_{F_\lambda} \lambda^b e_2\), with \(a+b=n\), and the quotient is
    \[
    \mathcal{L}/\mathcal{L}' \cong \dfrac{\mathcal{O}_{F_\lambda} e_1 + \mathcal{O}_{F_\lambda} e_2}{\mathcal{O}_{F_\lambda} \lambda^a e_1 + \mathcal{O}_{F_\lambda} \lambda^b e_2} \cong \dfrac{{\mathcal{O}_{F_\lambda}}}{\mathcal{O}_{F_\lambda} \lambda^a} \oplus \dfrac{\mathcal{O}_{F_\lambda}}{\mathcal{O}_{F_\lambda} \lambda^b},
    \]
    which has size \(|\mathbb{F}_\lambda|^{a+b} = \ell^{\f_\lambda n}\).

    The lattice \(\mathcal L\) is the \(\lambda\)-adic Tate lattice of an abelian variety \(A_0\) that is \(K\)-isogenous to \(A\), and the inclusion \(\mathcal L'\subseteq\mathcal L\) corresponds to a \(K\)-isogeny \(A'\to A_0\) whose kernel has order \(\ell^{\f_\lambda n}\). Since \(G_K\) acts trivially on \(\mathcal L/\mathcal L'\), this kernel is pointwise \(K\)-rational. Hence \(\Tors(A')(K)\) contains a subgroup of order \(\ell^{\f_\lambda n}\), as required.
\end{proof}

We remark that our assumption on the congruence of \(P_\p^\lambda(1)\) modulo \(\lambda^{n}\) implies the hypothesis in Katz's problem on the number of points of the reduction of \(A\) at \(\p\). 
In fact, as already mentioned, the condition \(N(\p) \equiv 0 \pmod{\ell^n}\) is equivalent to \(P_\p^{\ell}(1)= \det(1-\rho_\ell(\Frob_\p)) \equiv 0 \pmod {\ell^n }\), since \(N(\p)=P_\p^\ell(1)\). By \cite[\S 11.9-11.10]{Shimura1967AlgebraicGroups}, the characteristic polynomial of \(\rho_\ell(\Frob_\p)\) is equal to the norm from \(F\) to \(\Q\) of the characteristic polynomial of \(\rho_{\lambda}(\Frob_\p)\), for any \(\lambda\) over \(\ell\). Then, 
\begin{equation} \label{eq:valuations}
    v_{\ell}(P_\p^{\ell}(1)) = v_{\ell}(\mathcal{N}_{F/\Q}(P_\p^{\lambda}(1))) = \sum_{\lambda' \mid \ell} v_{\lambda'}(P_\p^\lambda(1)) \f_{\lambda'},
\end{equation}
where the sum varies over all primes \(\lambda'\) lying over \(\ell\) and \(\f_{\lambda'}\) is the residue degree of \(\lambda'\). Here \(P_\p^\lambda(1)\) is regarded as a single element of \(F\), and \(v_{\lambda'}\) is applied to this same element for every \(\lambda'\mid\ell\); we use the normalisation \(v_{\lambda'}(\lambda')=1\). See e.g.\ \cite[\S8]{Neukirch1999AlgebraicTheory} for the definition of \(v_{\lambda'}\) and its relation to \(v_\ell\), and for the proof that the norm \(\mathcal{N}_{F/\Q}\) is the product of the local norms. The equality follows immediately from these two facts.

If \(P_\p^\lambda(1) \equiv 0 \pmod {\lambda^n}\) then \(v_{\lambda}(P_\p^\lambda(1)) \geq n\) and so \(v_{\ell}(P_\p^{\ell}(1))\) is at least \(n\f_{\lambda}\). In particular, \(P_\p^{\ell}(1) \equiv 0 \pmod{\ell^n}\).

Let \(\Sigma\) be the set of primes \(\p\) of \(K\) of good reduction for \(A\) such that the absolute ramification index of \(\p\) is smaller than \(p-1\), where \(p\) is the rational prime below \(\p\). As mentioned in \S \ref{sec:intro}, by injectivity of the reduction map for the primes in \(\Sigma\), \(\displaystyle\max_{A'\text{ K-isogenous to }A} |\Tors(A')(K)|\) divides \(\displaystyle\gcd_{\p \in \Sigma} N(\p)\), and this bound is not always strict. However, in some cases, we can recover the exact result of Katz for elliptic curves and \(m\) equal to a prime power \cite[Theorem 2]{Katz1981GaloisVarieties}.

\begin{corollary}
    Let \(A\) be an abelian variety of \(\GL_2\)-type over a number field \(K\). Then, for every prime \(\ell\) such that there is only one prime \(\lambda\) in \(F\) above \(\ell\), we have that
    \(\ell\) divides \(\displaystyle\max_{A'\text{ K-isogenous to }A} |\Tors(A')(K)|\) if and only if it divides \(\displaystyle\gcd_{\p \in \Sigma} N(\p)\) and moreover
    \[
    v_\ell\left(\max_{A'\text{ K-isogenous to }A} |\Tors(A')(K)|\right)=v_\ell\left(\gcd_{\p \in \Sigma} N(\p)\right).
    \]
\end{corollary}

\begin{proof}
The result follows immediately from Theorem \ref{thm:main_new} and Equation \eqref{eq:valuations}.
\end{proof}

\begin{remark}
    Notice that the set \(\Sigma\) is infinite, since all but finitely many primes are of good reduction and unramified over \(\Q\) (so in particular \(\e_\p=1\)). However, in order to compute \(\gcd_{\p \in \Sigma} N(\p)\), only finitely many of these primes will be sufficient.
\end{remark}

\section{Modular abelian varieties defined over \texorpdfstring{$\Q$}{Q}}\label{sec:app}

We now specialise the result of Theorem \ref{thm:main_new} to simple \(\GL_2\)-type abelian varieties defined over the rational numbers. Let \(A\) be an abelian variety of dimension \(g\) defined over \(\Q\). Let \(\End^0_\Q (A) = \End_\Q(A) \otimes_\Z \Q\) denote the algebra of endomorphisms of \(A\) that are defined over \(\Q\), and let \(\End_{\overline{\Q}}^0(A) = \End_{\overline{\Q}} (A \times_\Q \overline{\Q}) \otimes_\Z \Q\) denote the full endomorphism algebra.

By \cite[p.\ 193]{Mumford1970AbelianVarieties}, if \(A\) is a geometrically simple abelian variety, both \(\End^0_\Q (A)\) and \(\End^0_{\overline \Q} (A)\) are Albert algebras. By \cite[Remark p.\ 182]{Mumford1970AbelianVarieties}, the dimensions of \(\End^0_\Q (A)\) and of \(\End_{\overline{\Q}}^0(A)\) divide \(2g\) and by \cite[\S 2]{Ribet2004AbelianForms}, if \(\End_\Q^0(A)\) contains a number field, then its degree divides \(g\).
In particular, by Albert's classification \cite[p.\ 201]{Mumford1970AbelianVarieties}, both endomorphism algebras are of one of the following types:
\begin{itemize}
    \item \textbf{Type I}: a totally real number field of degree dividing \(g\);
    \item \textbf{Type II or III}: a quaternion algebra (either totally definite or totally indefinite) over a totally real number field of degree dividing \(g/2\);
    \item \textbf{Type IV}: a division algebra of dimension \(d^2\) over a CM field of degree dividing \(2g/d^2\).
\end{itemize}

We now investigate the possible geometric structures of simple abelian varieties \(A\) of \(\GL_2\)-type over \(\Q\). As in \S \ref{sec:setting}, we write \(F \hookrightarrow \End_\Q^0(A)\), where \(F\) is a fixed degree \(g\) number field. By \cite[Theorem 2.1]{Ribet2004AbelianForms}, there is an isomorphism \(F \cong \End_\Q^0(A)\), hence from now on we will identify these two fields.

\begin{proposition}
    Let \(A\) be a simple abelian variety of dimension \(g \ge 2\), defined over \(\Q\), of \(\GL_2\)-type. Then one of the following occurs:
    \begin{itemize}
        \item \(A\) is not geometrically simple;
        \item \(A\) is geometrically simple and has potential quaternionic multiplication (or PQM for short);
        \item \(A\) is geometrically simple, with \(\End^0_{\overline{\Q}}(A) = \End^0_\Q(A)\) equal to a totally real number field of degree \(g\).
    \end{itemize}
\end{proposition}

\begin{proof}
Let us assume that \(A\) is geometrically simple, so that \(\End^0_{\overline{\Q}} (A)\) is an Albert algebra. In particular, it is either a division algebra or a field. Let \(Z\) be its center. By \cite[Proposition 5.1, Proposition 5.2, Theorem 5.3 and Proposition 3.6]{Ribet2004AbelianForms}, \(Z\) is totally real, contained in \(\End^0_{\Q} (A)\), and \(\End^0_{\overline{\Q}} (A)\) is either a quaternion algebra over \(Z\) (PQM case), or it is equal to \(Z\), in which case \(Z= \End^0_{\overline{\Q}} (A)= \End^0_{\Q} (A)\) as claimed.
\end{proof}

In the non-geometrically simple case, the study of \(\Tors(A)\) may be reduced to understanding the torsion of smaller dimensional abelian varieties, albeit defined over larger number fields. In the PQM case, for \(A\) of dimension \(2\), \cite{Laga2024RationalMultiplication} classifies all possible torsion subgroups. These first two cases are somewhat special: the generic case is that of \(\End^0_{\overline{\Q}}(A) = \End^0_\Q(A)\). Nonetheless, for what follows, we need not distinguish among these cases.

By \cite[Theorem 4.4]{Ribet2004AbelianForms} and Serre's modularity conjectures \cite{Serre1987SurmathrmGaloverlinemathbfQ/mathbfQ} (proved by Khare--Wintenberger, \cite{Khare2009Serresi,Khare2009Serresii}), a simple abelian variety \(A/\Q\) of \(\GL_2\)-type is modular, hence it is \(\Q\)-isogenous to the abelian variety \(A_f\) associated to some newform \(f\), with weight \(2\) and Nebentypus \(\chi\), and coefficient field equal to \(F=\End^0_\Q(A_f)\). The definition of \(A_f\) is in \cite[Definition 6.6.3]{Diamond2005AForms}. Such an abelian variety has \(\lambda\)-adic Galois representation isomorphic to that arising from the newform \(f\), which we denote by \(\rho_{f,\lambda}\). For the construction, see \cite[\S 9.5]{Diamond2005AForms}. By \cite[Theorem 9.6.5]{Diamond2005AForms}, the representations \((\rho_{f,\lambda})\) form a strictly compatible system for all \(\lambda\)'s, \emph{including those of bad reduction for \(A_f\)}; for every rational prime \(p\) not dividing the level \(N\), the characteristic polynomial of Frobenius is:
\[
P_p^\lambda(t) = 1 - a_p(f)t + \chi(p) p t^2,
\]
where \(a_p(f)\) is the \(p\)-th Fourier coefficient of \(f\).

This discussion allows us to restate Theorem \ref{thm:main_new} in the following form. The proof is identical to that in \S \ref{sec:main}.

\begin{theorem}\label{thm_mainQ}
Let \(A/\Q\) be a simple abelian variety of \(\GL_2\)-type, with an embedding \( F \hookrightarrow \End^0_{\Q}(A)\), and let \(f\) be a modular form such that \(A\) is \(\Q\)-isogenous to \(A_f\) and \(F\) is the field of coefficients of \(f\). Let \(\ell\) be a rational prime and \(\lambda\) be a prime of \(F\) lying over \(\ell\). Suppose that, for an integer \(n \ge 1\), the congruence
    \[
    \det(1-\rho_{f,\lambda}(\Frob_p)) \equiv 0 \pmod{\lambda^n}
    \]
    holds for a density-one set of primes \(p\). Then, there exists an abelian variety \(A'\) that is \(\Q\)-isogenous to \(A\) through a \(\ell\)-power isogeny, such that
    \[
    |\Tors(A')(\Q)| \equiv 0 \pmod{\ell^{\f_\lambda n}},
    \]
    where \(\f_{\lambda}\) is the residue degree of \(\lambda \mid \ell\).
\end{theorem}

\subsection{Implementation and conjectural results}\label{sec:implementation}
Theorem \ref{thm_mainQ} gives an algorithm for the computation of the torsion order of simple \(\GL_2\)-type abelian varieties over \(\Q\), which we implement in \texttt{Magma}, see \cite{Alessandri2026MagmaVarieties}.

More specifically, we access all\footnote{For information on the completeness of the data, see \url{https://www.lmfdb.org/ModularForm/GL2/Q/holomorphic/Completeness}} weight \(2\) newforms with level up to \(10000\) in \texttt{LMFDB} and list, for each, the following data:
\begin{itemize}
    \item the level \(N\);
    \item the Fourier coefficients \(a_p\) for rational primes \(p\) up to a chosen bound;\footnote{In the implementation we use the Sturm bound as a convenient finite cutoff for the Fourier coefficients. The Sturm bound determines a modular form from finitely many coefficients, but it does not imply that the gcd of the values \(P_p(1)\) over primes up to that bound has already stabilised. Empirically, a smaller cutoff was often enough.}
    \item the Nebentypus \(\chi\).
\end{itemize}
We use these data to compute, for the non-CM forms (in particular, the associated abelian varieties are non-CM), \(P_p^\lambda(1) = 1 - a_p + p\chi(p)\) for rational primes \(p\nmid N\) up to the bound and such that \(\e_p < p-1\), so that torsion injects into the reduction modulo \(p\). Notice that the polynomial \(P_p^\lambda(t) \in F_\lambda[t]\), which \emph{a priori} depends on \(\lambda\), is in fact in \(F[t]\) and it is independent of \(\lambda\); therefore we simply denote it by \(P_p(t)\). For the finite set \(S_f\) of rational primes \(p\) used in the computation (those not dividing the level, up to the chosen bound, and satisfying the reduction-injectivity condition), put
\[
n_\lambda(S_f)=\min_{p\in S_f}v_\lambda(P_p(1)).
\]
For each rational prime \(\ell\), we then compute
\[
\ell^{\,\max_{\lambda\mid\ell}\bigl(\f_\lambda n_\lambda(S_f)\bigr)},
\]
where the maximum is taken over the primes \(\lambda\) of \(F\) above \(\ell\). Multiplying these contributions over \(\ell\), we obtain the \emph{predicted torsion order}
\[
T_f(S_f)=\prod_{\ell\ \mathrm{prime}}\ell^{\,\max_{\lambda\mid\ell}\bigl(\f_\lambda n_\lambda(S_f)\bigr)}.
\]
If the finite minima \(n_\lambda(S_f)\) agree with the minima taken over all admissible primes, then Theorem \ref{thm_mainQ} identifies \(T_f(S_f)\) with
\[
\max_{A'\,\Q\text{-isogenous to }A_f}|\Tors(A')(\Q)|.
\]

The output of the code is 
\begin{itemize}
    \item the list of all possible predicted torsion orders found in this way, 
    \item the list of the torsion orders which are equal to the corresponding integer \(\gcd_{p \in S_f} N(p)\),
    \item the list of primes which occur in the factorisation of at least one of the possible torsion orders found. 
\end{itemize}
The predicted torsion order is not always equal to the gcd of the \(N(p)\)'s (computed over the same finite set \(S_f\)), hence Theorem \ref{thm_mainQ} cannot be improved to be exactly equivalent to Katz's result for elliptic curves.  Nonetheless, this equality is attained frequently. Within our dataset, the two lists produced by the code are exactly the same for abelian varieties of dimensions \(2\), \(3\), and \(5\), and in dimension \(4\), there is only one exception: \(88\) occurs as a predicted torsion order (namely for the newform with \texttt{LMFDB} label \texttt{\href{https://www.lmfdb.org/ModularForm/GL2/Q/holomorphic/4830/2/a/ce/}{4830.2.a.ce}}), but never appears in the second list.

We conclude the paper by listing the predicted torsion orders, including the divisors of the \(T_f(S_f)\), as these will also be torsion orders of some abelian variety, and the primes that divide them for non-CM simple abelian varieties of \(\GL_2\)-type and dimension up to \(5\), as computed by our code. The code and the list of newforms with the respective predicted torsion orders can be found at \cite{Alessandri2026MagmaVarieties}. We do not currently know whether the lists are exhaustive.

Code and data on how to associate in practice a classical modular form of weight \(2\), dimension \(2\), and trivial Nebentypus to all genus \(2\) curves of \(\GL_2\)-type whose Jacobians are \(\Q\)-isogenous to the modular abelian varieties associated to \(f\) is available at \cite{CostaModularAbelianSurfaces}.

    Let \(A/\Q\) be a simple non-CM \(g\)-dimensional abelian variety of \(\GL_2\)-type. The candidate orders for \(\Tors(A)(\Q)\) obtained from our finite computations are listed in Table \ref{tab:torsion}.
        \begin{table}[h!]
            \centering
            \begin{tabular}{r|l}
                \(g\) & \(|\Tors(A)(\Q)|\) \\
                \hline
                \(2\) & \(1, 2, 3, 4, 5, 6, 7, 8, 9, 10, 11, 12, 13, 14, 15, 16, 17, 
18, 19, 20, 21, 22, 23, 24, 28,\)\\ 
&\(31, 37, 44, 56\) \\
                \(3\) & \(1, 2, 3, 4, 5, 6, 7, 8, 9, 10, 11, 12, 13, 14, 15, 16, 17, 
18, 19, 20, 22, 23, 24, 28, 29,\)\\
&\(31, 32, 36, 37, 38, 40, 44, 46, 49, 58, 80, 83, 92, 160\)\\
                \(4\) & \(1, 2, 3, 4, 5, 6, 7, 8, 9, 10, 11, 12, 13, 14, 15, 16, 17, 
18, 19, 20, 21, 22, 23, 24, 25,\)\\
&\(26, 27, 28, 29, 31, 32, 34, 35, 36, 37, 38, 40, 41, 43, 44, 48, 49, 52, 55, 56, 57, 61,\)\\
&\(
63, 64, 68, 71, 72, 73, 74, 76, 80, 82, 88, 136, 146, 152, 155, 176\)\\
                \(5\) & \(1, 2, 3, 4, 5, 6, 7, 8, 9, 10, 11, 12, 13, 14, 16, 17, 18, 
19, 20, 21, 22, 23, 24, 25, 26,\)\\
&\(27, 28, 29, 31, 32, 34, 35, 36, 38, 40, 41, 43, 44, 46, 48, 50, 52, 53, 55, 56, 63, 64,\)\\
&\(68, 72, 80, 92, 96, 106, 110, 144, 160, 192, 288, 320\)\\
            \end{tabular}
            \caption{Torsion orders}
            \label{tab:torsion}
        \end{table}

For the reader's convenience, we also list, in Table \ref{tab:prime}, the primes \(\ell\) that show up in our computations as divisors of the torsion order of simple non-CM \(g\)-dimensional abelian varieties of \(\GL_2\)-type.

        \begin{table}[h!]
            \centering
            \begin{tabular}{r|l}
            \(g\) & \(\ell\) \\
            \hline
        \(2\) & \(2, 3, 5, 7, 11, 13, 17, 19, 23, 31, 37\)\\
         \(3\) & \(2, 3, 5, 7, 11, 13, 17, 19, 23, 29, 31, 37, 83\)\\
        \(4\) & \(2, 3, 5, 7, 11, 13, 17, 19, 23, 29, 31, 37, 41, 43, 61, 71, 73 \)\\
         \(5\) & \(2, 3, 5, 7, 11, 13, 17, 19, 23, 29, 31, 41, 43, 53\)\\
    \end{tabular}
     \caption{Primes}
            \label{tab:prime}
        \end{table}

\section*{Acknowledgments}
The authors are grateful to Martin Azon, Sander Dahmen, Sam Frengley, Julian Demeio, Pip Goodman, Davide Lombardo, Laura Paladino and Matteo Verzobio for contributing to making this collaboration possible and for their useful insights.
\section*{AI usage}
The authors used OpenAI's ChatGPT (GPT-6) to review the exposition and suggest improvements to the wording. The authors reviewed the suggestions and take full responsibility for the final content of the manuscript.
\section*{Funding}
The authors are members of the INdAM group GNSAGA. JA was supported by UKRI Future Leaders Fellowship \texttt{MR/V021362/1} and the Max Planck Institute for Mathematics in Bonn, which she thanks for its hospitality and financial support. NC is supported by the \emph{``Piano di Sviluppo Dipartimentale''} with the title \emph{``Enhancing the Research in the Math Dept'', CUP C93C24000160005}. NC thanks the University of Bath for its hospitality during her visit in May 2025, where most of this work was carried out. 
\bibliographystyle{alpha}
\bibliography{refs}
\end{document}